\documentclass[a4paper,12pt]{article}
\usepackage{amsfonts}
\usepackage{amsmath}
\usepackage{amssymb}
\usepackage{amstext}
\usepackage{amsthm}
\usepackage{xspace}
\usepackage{graphicx}

\newcommand{\R}{\mathbb R}
\newcommand{\N}{\mathbb N}
\newcommand{\Q}{\mathbb Q}
\newcommand{\Z}{\mathbb Z}

\newtheorem{theorem}{Theorem}
\newtheorem{lemma}[theorem]{Lemma}
\newtheorem{prop}[theorem]{Proposition}
\newtheorem{cor}[theorem]{Corollary}

\theoremstyle{definition}
\newtheorem{definition}[theorem]{Definition}
\newtheorem{rem}[theorem]{Remark}

\theoremstyle{remark}

\begin{document}
\title{Periodic orbits of a one dimensional non autonomous Hamiltonian
system}
\author{J.Bellazzini, V.Benci, M.Ghimenti}
\date{}
\maketitle
\begin{center}
\small Dipartimento di Matematica Applicata \\
\small Universit\`a di Pisa\\
\small Via Bonanno Pisano 25/b 56126 PISA -Italy\\ 
\end{center}
\begin{abstract}
In this paper we study the properties of the periodic orbits of 
$\ddot{x}+V'_{x}(t,x)= 0$ with $x \in S^1$ and $V'_{x}(t,x)$ a 
$T_0$ periodic potential.
Called $\rho\in\frac{1}{T_0}\Q$ the frequency of windings of an orbit in $S^1$ 
we show that exists an infinite number
of periodic solutions with a given $\rho$. We give a lower bound on the
number of periodic orbits with a given period and $\rho$ 
by means of the Morse theory.
\end{abstract}

\section{Introduction}
In this paper we study the second order Hamiltonian system
\begin{equation}
\label{equation1}
\ddot{x}+V'_{x}(t,x)= 0
\end{equation}
where $x \in S^1=\R /\Z$ and $V \in C^2(\mathbb{R} \times S^1)$ is
a periodic potential with minimal period  $T_0$. 

There are two question that we study in this paper. First, we study
the existence of periodic 
solutions of (\ref{equation1}) 
in any connected component of the space of periodic trajectories, 
i.e in the space of trajectories
that makes $k_1$ windings in $S^1$ in $k_2T_0$ time  with $k_1$ and $k_2$ 
arbitrary integers that are coprime.

Second, called $\rho(x)=\frac{k_1}{k_2T_0}$ the frequency of 
windings of $x(t)$ in 
$S^1$, we study the existence of orbits with the same $\rho$ that are not
$k_2T_0$ periodic, i.e
periodic orbits that  make $mk_1$ windings in $mk_2T_0$ time $m \in \N$, with $k_1$
and $k_2$ coprime, when the solutions are not $k_2T_0$ periodic.

The problem of the search of periodic orbits is classical 
and a standard approach to these problems is that of studying the critical 
points of the action functional
\begin{equation}\label{functional}
f(x)=\frac{1}{k_2T_0}\int_{0}^{k_2T_0}\left(\frac{1}{2}|\dot{x}|^2-V(t,x)\right)dt
\end{equation}
in the space of functions that makes 
$k_1$ windings in $k_2T_0$ time.

This method have been largely used by many people in the last
twenty years: see e.g. the book of Rabinowitz \cite{Rab02} and the
references therein. 

The problem of the search of periodic orbits with a given frequency of 
windings in $S^1$ is closely related to that of the existence of 
subharmonic orbits.
In particular, for what concern the existence of 
subharmonic orbits, we quote \cite{Sal92} for a general 
Hamiltonian system on $S^1\times T^{2n}$
under non-degeneracy condictions, \cite{Lon00} for the existence of infinitely many subharmonics for more general Lagrangian systems and 
\cite{Ser00} for the case of a second order differential equation in 
$\R^{n}$ with a time dependent periodic 
potential and a periodic forcing term with zero mean value. 

In his paper, the existence  of periodic orbits of
equation (\ref{equation1}) has been studied by means of the Morse 
relations applied to the action functional. We will show that the Morse 
relations allows to prove the existence of infinitely many orbits 
with a  given $\rho=\frac{k_1}{k_2T_0}$ and to give a lower bound on 
the number of periodic orbits with a given period.

\section{Statements of the results}

We set
$$ C^2_{k_1,k_2T_0}=\{[x]:x \in C^2( \R , \R ),\ x(t+k_2 T_0)=x(t)+k_1\};$$
where $[x]=x  \mod  1$;  namely $ C^2_{k_1,k_2T_0}$ is the
space of the periodic $ C^2$-functions which make $k_1$ windings
in $k_2T_0$
time; thus we have that $C^2_{k_1,k_2T_0} \subset C^2_{mk_1,mk_2T_0},\ m\in \N^+$.
Given a periodic orbit, $x(t)$, the rotation frequency $\rho=\rho(x)$
associated to $x(t)$ can be defined as the frequency of windings of the
periodic orbit in $S^1$, i.e. the number of windings
divided by $k_2T_0$.  Given $k_1 \in
\mathbb{Z}$ and $k_2 \in \mathbb{N}$, the periodic orbit that makes
$k_1$ windings in $k_2T_0$ time has a rotation frequency  
$\rho=\frac{k_1}{k_2T_0}$.
Clearly, the set of periodic functions with rotation frequency $\rho=\frac{k_1}{k_2T_0}$ when $k_1$ and $k_2$ are coprime
is given by $\bigcup_{m=1}^{\infty} C^2_{mk_1,mk_2T_0}$.

\begin{definition}
A periodic solution $x(t)$ of rotation frequency $\rho=\frac{k_1}{k_2T_0}$ with $k_1$ and $k_2$ coprime
is called a {\em fundamental solution}
if $x\in C^2_{k_1,k_2T_0}$. Otherwise, if  $x \notin C^2_{k_1,k_2T_0}$,
it is called {\em non-fundamental}.
\end{definition}

\begin{definition}
A periodic solution $x(t)$ is called {\em non-resonant} if the linearized equation
$\ddot{y}+V''(t,x(t))y=0$
has no periodic solution (different from 0); the equation (\ref{equation1}) is called
{\em non-resonant} if all its periodic solutions are non-resonant.
\end{definition}

From now on we assume that eq.(\ref{equation1}) is non-resonant.
This is a technical assumptions which makes easier to use Morse theory.
In fact, if $x$ is a non-resonant T-periodic solutions, it is a non-degenerate
critical point of (\ref{functional}).\\

The first result of this paper is the following

\begin{theorem}\label{primoteo}
If eq. (\ref{equation1}) is non-resonant, for every $\rho=\frac{k_1}{k_2T_0}$, with $k_1$ and $k_2$ coprime, equation
(\ref{equation1}) has exactly $2r$ fundamental  solutions with
$r>0$ and infinitely many non-fundamental solutions.
\end{theorem}

Clearly any $k_2T_0$-periodic solution $x(t)$ is also a
$mk_2T_0$-periodic solution, $m\in\N$; thus $x(t)$ is a critical point of the
functional (\ref{functional}) with $T=k_2T_0$  and $T=mk_2T_0$
respectively, and the Morse index $m(x,T)$ is well defined for
such values of $T$. Given a periodic orbit $x(t)$, we define the twisting frequency (also called the twisting number or the mean index) $\tau$ as the
mean Morse index, i.e.
$\lim\limits_{T\rightarrow \infty} \frac{m(x,T)}{T}.$

The second result of this paper concerns the number of 
non-fundamental solutions in $C^2_{pk_1, pk_2T_0}$ with
$p$ prime.

We introduce two function $\nu(\tau, \rho)$ and $\eta(\tau, \rho)$, that are 
related to
the number of fundamental solutions with twisting frequency less than $\tau$ and
rotation frequency $\rho$. The value of such functions permits to give a 
lower bound on the number of non-fundamental orbits with rotation frequency $\rho$ in 
$C^2_{pk_1, pk_2T_0}$.

In order to state the main theorem we need to to classify the periodic
orbits in two classes, the
class $\alpha$ of periodic orbits with even Morse index and the
class $\beta$ with odd Morse index.

As we will see, the $\alpha$-periodic orbits are those  with
distinct
positive Floquet multipliers while the $\beta$-periodic are those 
possessing negative or complex Floquet multipliers. We
set for any $\tau \in \mathbb{R}^+$ and $\rho \in \frac{1}{T_0}\mathbb{Q}$

\begin{eqnarray*}
n_\alpha(\tau,\rho)&=&
\left\{
\begin{array}{c}
\text{number of fundamental solutions } x \text{ of type }\alpha
\text{ with }\\
\tau(x) =\tau ;\ \rho(x)=\rho
\end{array}
\right\} \\
n_\beta(\tau,\rho)&=&
\left\{
\begin{array}{c}
\text{number of fundamental solutions } x \text{ of type }\beta
\text{ with }\\
\tau(x) =\tau ;\ \rho(x)=\rho
\end{array}
\right\}
\end{eqnarray*}
we can define the function
\begin{displaymath}
\chi (\tau,\rho) :=n_\alpha(\tau,\rho)-n_\beta(\tau,\rho)
\end{displaymath}
and the multiplicity functions
\begin{eqnarray*}
\nu (\tau,\rho) :=\sum_{\zeta \leq \tau }\chi (\zeta,\rho);&&
\eta (\tau,\rho) :=\sum_{\zeta < \tau }\chi (\zeta,\rho).
\end{eqnarray*}

By this functions we can prove the second results of this paper

\begin{theorem}
\label{maintheorem}

Let $\rho = k_1/(k_2T_0)$ with $k_1$ and $k_2$ coprime, $p$ prime, and
assume that eq.(\ref{equation1}) is non-resonant.

Then, the periodic solutions in $C^2_{pk_1,pk_2T_0}$
having
rotation frequency $\rho$, and Morse index $2n$ are of type $\alpha$ and are
at least 
$$\nu \left( \frac{2n}{pk_2T_0}, \rho \right) \mod p.$$

Moreover, the periodic
solutions in $C^2_{pk_1, pk_2T_0}$
having rotation frequency $\rho$, and Morse index $2n+1$ are of type $\beta$
and are at least
$$ - \eta \left( \frac{2n+2}{pk_2T_{0}}, \rho \right) \ \mod p . $$
\end{theorem}

As a consequence, we have the following corollary.

\begin{cor}\label{coroll}
We set
\begin{equation}
\Sigma=\overline{\{(\tau,\rho)\ :\ \nu(\tau,\rho)\neq0\}}.
\end{equation}
Then, for any $(\tau,\rho) \in \Sigma$, there exists a sequence of non-fundamental solutions $\{x_n\}$ of type $\alpha$ and
a sequence of solution $\{y_n\}$ of type $\beta$ such that
\begin{eqnarray*}
\rho(x_n)\rightarrow \rho&&\tau(x_n)\rightarrow \tau\\
\rho(y_n)\rightarrow \rho&&\tau(y_n)\rightarrow \tau
\end{eqnarray*}
\end{cor}

\section{The Morse relations}

In order to obtain some estimates on the number of critical point,
we must recall some features of Morse theory. After
a short summary of the main results, we show some preliminary lemma useful
to apply Morse theory to our framework. For an exhaustive treatment of
Morse theory, and for the proofs of the results here
collected the reader can check \cite{Bot82}, \cite{Pal63},
\cite{Ben91,BG94,Ben95}.
\begin{definition}
Let $M$ a $C^2$ complete differential manifold and
let $f\in C^2(M,\R)$ function. Let $x\in M$ a critical point of $f$.
Suppose that $x$ is non-degenerate, i.e. the Hessian determinant
does not vanish in $x$.

Then the Morse index $m(x)$ is the signature of the Hessian of $f$ at $x$
\end{definition}

By this definition it is possible to prove the following theorem.

\begin{theorem}
Let $M$ be a complete $C^2$ Riemannian manifold, $f\in C^2(M,\R)$.
Set
\begin{eqnarray}
f^b&=&\{x\in M\ :\ f(x)\leq b\};\\
f_a^b&=&\{x\in M\ :\ a\leq f(x)\leq b\};
\end{eqnarray}
let $c\in\R$ be
the unique critical level in the interval $[a,b]$. Suppose that the
critical points in $f^{-1}(c)$ are non-degenerate
and suppose that
$f_a^b$ is a compact set in $M$. If there are $n$ critical
points of index $q$ at level $c$, then
\begin{equation}
\dim H_q(f^b,f^a)=n,
\end{equation}
where $H_*(X,Y)$ is the $\Z$ singular homology of the couple.
\end{theorem}

We must introduce now
the Poincar\'e polynomial; this algebraic tool allows us to formulate
the main theorem of this paragraph.
\begin{definition}
Let $(X,A)$ be a topological pair. Then the
Poincar\'e polynomial ${\cal P}_\lambda(X,A)$ is the formal series
in the $\lambda$ variable
with non negative integer coefficients (maybe infinite) defined by
\begin{equation}
{\cal P}_\lambda(X,A)=\sum_{q\in\N}\dim H_q(X,A){\lambda}^q.
\end{equation}
Moreover ${\cal P}_\lambda(X):={\cal P}_\lambda(X,\emptyset)$.
\end{definition}

At last, we can state the so called {\em Morse relations}, that are useful
to estimate the number of critical points of a function.

\begin{theorem}[Morse relations]
Let $M$ be a complete $C^2$ Riemannian
manifold, $f\in C^2(M,\R)$, and let $a,b$ be two regular
values of $f$. If $f_a^b$ is compact and all the critical points are
nondegenerate, then
\begin{equation}
\sum_{x\text{critical in }  f_a^b}\lambda^{m(x)}=
{\cal P}_\lambda(f^a,f^b)+\left(1+\lambda\right){\cal Q}_\lambda.
\end{equation}
If also $M$ is compact, then
\begin{equation}
\label{morserelationsclass}
\sum_{x\text{critical}}\lambda^{m(x)}=
{\cal P}_\lambda(M)+(1+\lambda) {\cal Q}_\lambda.
\end{equation}
where $m(x)$ is the Morse index of $f$ at $x$, and
${\cal Q}_\lambda$ is a formal series with non negative integer coefficients.
\end{theorem}

If $f^b_a$ is not compact, the above theorem is no longer valid.
This assumption can be substituted with the Palais Smale compactness
condition, recalled below. This condition permits to extend Morse relations
when, as in our case, $M$ is a non compact infinite dimensional manifold.
\begin{definition}
Let $H$ be an Hilbert space; $f\in C^1(H,\R)$ satisfies the $(PS)_c$ condition
iff every sequence $\{u_h\}_h\subset H$ s.t.
\begin{eqnarray*}
&&||\nabla f(u_h)||\rightarrow0;\\
&&f(u_h)\rightarrow c,
\end{eqnarray*}
is relatively compact in $H$.
\end{definition}

\section{Variational settings}

Now let us consider the dynamical system defined by equation
(\ref{equation1}):
\begin{equation*}
\ddot{x}+V'_{x}(t,x)= 0
\end{equation*}
where $x \in S^1$, $V \in C^2(\mathbb{R} \times S^1)$ and $V'$
denotes the derivative
of $V$ with respect to $x$. We suppose that $V(t,\cdot)$ is $T_0$-periodic.\\
We introduce three different spaces:
\begin{displaymath}
H^1_{k_2T_0} = \{ x: x \in H^1_{loc}(\mathbb{R},S^1), \ x(t+k_2T_0)=x(t)\},
\end{displaymath}
the Hilbert space of all the periodic orbits with period $k_2T_0$,
equipped with the following scalar product
\begin{displaymath}
<u,v>=\frac{1}{k_2T_0}\int_0^{k_2T_0} (\dot{u} \cdot \dot{v}+u \cdot v) dt,
\end{displaymath}
with $u$, $v \in TH^1_{k_2T_0}=H^1_{k_2T_0}$,
\begin{displaymath}
H^1_{0,k_2T_0} = \{ [x]:\ x \in H^1_{loc}(\mathbb{R},\mathbb{R}), \ x(t+k_2T_0)=x(t)\}
\end{displaymath}
the Hilbert space of the periodic orbits with period $k_2T_0$ in the 0-th connect component,
and
\begin{displaymath}
H^1_{k_1,k_2T_0} = \{ [x]:\  x \in
H^1_{loc}(\mathbb{R},\mathbb{R}), \ x(t+k_2T_0)=x(t)+k_1\}
\end{displaymath}
the set of $k_2T_0$ periodic orbits that make $k_1$ windings in
$S^1$, where $[x]=x \mod  1$. We have that
$H^1_{k_1,k_2T_0}$ is an Hilbert affine space. Indeed, given
$x(t) \in H^1_{k_1,k_2T_0}$, there exist $y(t) \in H^1_{0,k_2T_0}$ such that
\begin{equation}\label{bound}
x(t)=\frac{k_1}{k_2T_0}t+y(t).
\end{equation}

We are interested to study the $k_2T_0$-periodic solution of (\ref{equation1}).
The equation (\ref{equation1}) is the
Euler-Lagrange equation corresponding to the functional
\begin{equation}
\label{funzionale}
f(x)=\frac{1}{k_2T_0}\int_{0}^{k_2T_0}(\frac{1}{2}|\dot{x}|^2-V(t,x))dt
\end{equation}
 on $H^1_{k_1,k_2T_0}$ and the $k_2T_0$-periodic solutions of equation (\ref{equation1}) are the
critical point of the functional (\ref{funzionale}). It is well known that
the functional is $C^2$ on $H^1_{k_1,k_2T_0}$ and we can apply the Morse theory
defining a Morse index for every $k_2T_0$-periodic solution of
(\ref{equation1}). If $x(t)$ is a $k_2T_0$-periodic solutions of equation (\ref{equation1}) then
\begin{equation}
\label{funzionale'}
f'(x)[y]=\frac{1}{k_2T_0}\int_{0}^{k_2T_0}(\dot{x}\cdot
\dot{y}-V'(t,x)y)dt=0
\end{equation} for all $y \in
H^1_{0,k_2T_0}$ with $y(0)=0$. The Hessian of the functional $f$ is defined as
\begin{equation}
\label{funzionale''}
f''(x)[y][y]=\frac{1}{k_2T_0}\int_{0}^{k_2T_0}(|\dot{y}|^2-V''(t,x)
y^2)dt
\end{equation}
and the signature of the Hessian at $x(t)$ is given by the number
of negative eigenvalues of (\ref{funzionale''}).
\begin{definition}
We denote $m(x,k_2T_0)$ the Morse index relative to the $k_2T_0$
periodic orbit $x(t)$, i.e. the signature of (\ref{funzionale''}).
\end{definition}
\subsection{Poincar\'e polynomial of the free loop space}

Now we can compute the Poincar\'{e} polynomial of the
$H^1_{k_1,k_2T_0}$, that is the $k_1$-th connected component of
the path space $ H^1_{k_2T_0}$.

We recall some feature of the Poincar\'{e} polynomial that we need
to prove our result. For all the details and for an exhaustive
treatment of the Poincar\'e polynomial
we refer to \cite{Ben91,BG94}.
Here we recall only the following standard result of algebraic topology.
\begin{rem}\label{plom}
Let $(X,A)$ and $(Y,B)$ be two pairs of topological spaces. Then
\begin{enumerate}
\item{if $(X,A)$ and $(Y,B)$ are homotopically equivalent, then
${\cal P}_\lambda(X,A)={\cal P}_\lambda(Y,B)$;}
\item{${\cal P}_\lambda(X\times Y,A\times B)=
{\cal P}_\lambda(X,A)\cdot{\cal P}_\lambda(Y,B)$ (K\"unnet formula);}
\item{if $x_0$ is a
single point then ${\cal P}_\lambda(\{x_0\})=1$; furthermore if $X$ is
topologically trivial even ${\cal P}_\lambda(X)=1$;}
\end{enumerate}
\end{rem}

It is obvious that $H^1_{k_2T_0}\simeq H^1(S^1,S^1)$, and so also
$H^1_{k_1,k_2T_0}\simeq H^1_{k_1}(S^1,S^1)$.
Furthermore,
in order to calculate the Poincar\'{e} polynomial of the path space,
by the Whitney theorem we know that there is an homotopic equivalence
between $H^1(S^1,S^1)$
and $C^0(S^1,S^1)$. This is a standard argument, and can be found, for example,
in \cite{Kli78}.
Thus, we can consider the $k$-th connected component of
$C^0(S^1,S^1)$ that is the
set of continuous maps from $S^1$ to $S^1$ with index $k$
(roughly speaking the curves that "turns" $k$ times around $S^1$).
We note this component as $C^0_k(S^1,S^1)$. We want to show the following Lemma

\begin{lemma}
\label{basiclemma}
For all integer $k_1 \in \mathbb{Z}$, we have that
\begin{displaymath}
{\cal P}_\lambda(H^1_{k_1,k_2T_0})=1+\lambda
\end{displaymath}
\end{lemma}

\begin{proof}

We have just said that
$H^1_{k_1,k_2T_0}\simeq H^1_{k_1}(S^1,S^1)\simeq C^0_{k_1}(S^1,S^1)$.
Now it's easy to see that
\begin{eqnarray*}
C^0_{k_1}(S^1,S^1)&\simeq&
\left\{u:
u\in C^0(\R,\R), \ u(t+k_2T_0)=u(t)+k_1 \right\}\simeq
\\
&\simeq&S^1\times
\left\{u:
u\in C^0(\R,\R), \ u(t+k_2T_0)=u(t)+k_1,\ u(0)=0
\right\}.
\end{eqnarray*}
The space $\left\{u:
u\in C^0(\R,\R), \ u(t+k_2T_0)=u(t)+k_1,\ u(0)=0
\right\}$ is an affine space, so it is contractible and its Poincar\'{e} polynomial
is equal to 1.
Then, by the K\"unnet formula we obtain
\begin{equation}
{\cal P}_\lambda(C^0_{k_1}(S^1,S^1))={\cal P}_\lambda(S^1)=1+\lambda,
\end{equation}
that concludes the proof.
\end{proof}

\subsection{The Palais Smale condition}

We show now that the functional
\begin{displaymath}
f(x)=\frac 1{T_0}\int\frac 12 |\dot x|^2-V(t,x)\ dt
\end{displaymath}
defined at the beginning of this section satisfies the Palais Smale condition.
The result is well known because the potential is bounded and we prove it in
the standard way.

\begin{prop}
The functional $f$ satisfies the (PS) condition in $H^1_{k_1,k_2T_0}$
\end{prop}
\begin{proof}
At first we notice that $V$ is bounded. In fact, $V$ is $C^2$,
periodic in the $t$ variable, and $x(t)$ is periodic.
Furthermore, because $x\in H^1_{k_1,k_2T_0}$, is also continuous, so
the potential $V$ is bounded.

Suppose that $x_n$ is a Palais Smale sequence, i.e. that
\begin{equation}\label{PS1}
f(x_n)=\frac 1{T_0}\int\frac 12 |\dot x_n|^2-V(t,x_n)\ dt\rightarrow c\in\R;
\end{equation}
\begin{equation}\label{PS2}
f'(x_n)[v]=\frac 1{T_0}\int \dot x_n\dot v -V'(t,x_n)v\ dt \rightarrow 0\ \forall v\in H^1_{0,k_2T_0}.
\end{equation}
By (\ref{PS1}), we know that $f(x_n)$ is bounded.
Because $V(t,x_n)$ is bounded, we have that
also $||x_n||_{H_1}$ is bounded.
Thus, up to subsequence, $x_n\rightharpoonup x$
weakly in $H^1$, furthermore, for the
Sobolev immersion theorem, we have that
$x_n\rightarrow x$ in $L^2$ and uniformly. By (\ref{PS2}) we have that
\begin{equation}
f'(x_n)[x_n-x]=\frac 1{T_0}\int <\dot x_n,\dot x_n-\dot x>  -V'(t,x_n)(x_n-x)\ dx \rightarrow 0.
\end{equation}
We know that $V'(t,x_n)\rightarrow V'(t,x)$ uniformly (and thus $L^2$).
Then
\begin{equation}
\int V'(t,x_n)(x_n-x)\rightarrow0.
\end{equation}
So we obtain that
\begin{equation}
\frac 1{T_0}\int <\dot x_n,\dot x_n-\dot x> =
\frac 1{T_0}\int |\dot x_n|^2-\frac 1{T_0}\int \dot x_n\dot x\rightarrow 0.
\end{equation}
But, because $x_n\rightarrow x$ weakly in $H^1$ we have that
\begin{equation}
\int \dot x_n\dot x\rightarrow\int |\dot x|^2,
\end{equation}
so we have that
\begin{equation}
||x_n||_{H^1}\rightarrow||x||_{H^1},
\end{equation}
that concludes the proof.
\end{proof}

\section{The Bott and Maslov indexes}
The Morse index of a periodic orbit
is  strongly related to two others indexes.
One is the Maslov index and the other is an index
that we have called Bott index since it has been introduced in the study of
geodesics by Bott.

These indexes turn out to have the same numerical value but they refer to
different mathematical objects. Indeed, the Morse index of a periodic orbit
$x(t)$ measures the signature of the Hessian of $f$ at $x(t)$, the Maslov index
the half windings in the symplectic group $Sp(2)$ of the matrix of the
fundamental solutions of the linearized equation around $x(t)$
and the Bott index the negative eigenvalues of the
operator $-\ddot{y}-V''(x(t),t)y$ associated to the linearized equation.

We need to introduce the Bott index to easily compute the twisting frequency
of a periodic orbit while the Maslov index to characterize the periodic orbits
of type $\alpha$ and type $\beta$.

\subsection{The Bott index and the twisting frequency}
Let us consider, for $\sigma \in S^1=\{ z \in \mathbb{C} : |z|=1\}$,
\begin{displaymath}
L^2_{\sigma,T}=\{ x \in L^2_{loc}(\mathbb{R},\mathbb{C^N})\
:\ x(t+T)=\sigma \cdot x(t)\text{ for a.a. }t\in\mathbb{R} \}
\end{displaymath}
where $L^2_{loc}$ is the set of function $x:\mathbb{R} \rightarrow \mathbb{C^N} $
which are measurable and whose square is locally integrable.
$L^2_{\sigma,T}$ is an Hilbert space with the following scalar product
\begin{displaymath}
(x,y)=\frac{1}{T}\int_0^T (x(t),y(t))_{\mathbb{C^N}}dt.
\end{displaymath}
Now, we consider the following differential equation
\begin{equation}
\label{eigenvalue}
\ddot{y}+A(t)y=-\lambda y,
\end{equation}
with $y \in \mathbb{C^N},\lambda \in \mathbb{R} $ and $A(t)$ a family of real
symmetric $N \times N$ matrices $T_0$-periodic,
defined on $L^2_{\sigma,T_0}$.\\
Let $W^2_{loc}(\mathbb{R},\mathbb{C^N})$ be the space of functions having two
square locally integrable derivative and ${\cal{L}}_{\sigma,T_0}$ be the
extension
to  $W^2_{loc}(\mathbb{R},\mathbb{C^N}) \cap L^2_{\sigma,T_0}$ of the operator
\begin{displaymath}
-\ddot{y} - A(t)y.
\end{displaymath}
The eigenvalue problem (\ref{eigenvalue}) becomes
\begin{equation}
{\cal{L}}_{\sigma,T_0}y = \lambda y,
\end{equation}
with $y \in W^2_{loc}(\mathbb{R},\mathbb{C^N}) \cap L^2_{\sigma,T_0}$.
The spectrum of this selfadjoint unbounded operator is discrete
with a finite number of negative eigenvalue.\\
This fact allows us to define a function
\begin{displaymath}
j(T_0, \cdot): S^1 \rightarrow \mathbb{N}
\end{displaymath}
as follows:
\begin{equation}
j(T_0,\sigma)=
\left\{
\begin{array}{c}
\text{ number of negative eigenvalues of }{\cal{L}}_{\sigma,T_0}\\
\text{ counted with their multiplicity.}
\end{array}
\right\}
\end{equation}
In order to define the Bott index we need that the operator ${\cal{L}}_{1,T_0}$
is nondegenerate, i.e that $0$ is not an eigenvalue of ${\cal{L}}_{1,T_0}$ .

In this case we can define the Bott index in the
following way:
\begin{definition}
We denote the function $j(T_0,1)$ the Bott index relative to the equation
$\ddot{y}+A(t)y=0$ in the interval $[0,T_0]$.
\end{definition}
Now let $W(t): \mathbb{C}^{2N} \rightarrow \mathbb{C}^{2N} $ the matrix
of the fundamental solutions relative to the equation $\ddot{y}+A(t)y=0$,
namely the solution of the following Cauchy problem
$$
\left \{
\begin{array}{ll}
\dot{W}(t) + \mathcal{A}(t)W(t)=0 \\
W(0)= I.
\end{array}
\right.
$$
where
$$
\mathcal{A}(t) = \left(
\begin{array}{cc}
0 & I \\
-A(t) & 0\\
\end{array}
\right).
$$
The eigenvalues of $W(T_0)$ are called Floquet multipliers.
The nondegenerate condition means that the linear system
 $\ddot{y}+A(t)y=0$ does not have any nontrivial $T_0$-periodic solutions, i.e
that  $1$ is
not a Floquet multiplier of $W(T_0)$.\\
The Bott index fulfills the following properties

\begin{prop}
\label{propMaslov}
The function $j(T_0,\sigma)$ satisfies the following properties.\\
(i) $j(T_0,\sigma)= j(T_0,\bar{\sigma})$\\
\\
(ii) if $j(T_0,\sigma)$ is discontinuous at the point $\sigma ^*$
then $\sigma ^*$ is a Floquet multiplier\\
\\
(iii) $|j(T_0,\sigma_2)-j(T_0,\sigma_1)| \leq l \quad \forall \sigma_2,\sigma_1 \in S^1- \{+1,-1\} $
where $2l$ is the number of non-real Floquet multipliers on $S^1$
counted with their molteplicity\\
\\
(iv) $$j(kT_0,\theta)=\sum_{j=0}^{k-1}j(T_0,\sigma_j)$$ where $\sigma_0,\sigma_1,...,\sigma_{k-1}$
are the $k$ values of $\sqrt[k]{\theta}$.
\end{prop}
The proof of $(i)$, $(ii)$, $(iii)$, $(iv)$ is contained in \cite{Ben91}.\\

The Bott index allows to define the twisting frequency as follows:
\begin{equation}
\tau = \frac{1}{2\pi T_0} \int_{0}^{2 \pi} j(T_0,exp[i\omega]) d\omega.
\end{equation}

\begin{prop}
\label{propTwist}
The twisting frequency satisfies the following properties:\\
(i) $\tau$ = $\lim_{T \rightarrow \infty}$ $\frac{1}{T}j(T,1)$ \quad $T=kT_0$\\
\\
(ii) $\tau = \frac{1}{2\pi T} \int_{S^1} j(T,\sigma) d\sigma$  \quad
$T=kT_0$\\
\\
(iii) $|T\tau - j(T,\sigma)| \leq l$ \quad  $\forall \sigma \in S^1- \{+1,-1\} $ where $2l$ is the
number of non-real Floquet multipliers on $S^1$ counted with their multiplicity and $T=kT_0$\\
\\
(iv) $\forall \sigma \in S^1$ we have  $\tau$ = $\lim_{T \rightarrow \infty}$
$\frac{1}{T}j(T,\sigma)$ \quad $T=kT_0$\\
\end{prop}
The proof of $(i)$, $(ii)$, $(iii)$, $(iv)$ is contained in \cite{Ben91}.

\subsection{The Maslov index and the geometrical representation of $Sp(2)$}
In this section we give some properties of the Morse index 
by means of Maslov index in the two dimensional case.

Let us consider the linear equation
\begin{equation}
\label{linearequation}
\ddot{y}+A(t)y=0
\end{equation}
where $A(t)$ is $T_0$-periodic.
Let $W(t)$ be the matrix of the fundamental solutions of the linear equation
(\ref{linearequation}) at time $t$, with $t \in [0,T]$. The matrix  $W(t)$
is unimodular, i.e it is symplectic and we can associate to
the linear equation (\ref{linearequation}) a path $\gamma$ in the symplectic group.
The Maslov index is an integer associated to
the path of  $W(t)$ in the symplectic group. The Maslov index theory for 
any non degenerate path in $Sp(2)$ was established first in \cite{CZ84} and 
\cite{Lon90}; we avoid rigorous definitions for the sake of 
brevity and we refer to the book of Abbondandolo \cite{Abb01}. 

 Loosely speaking, the Maslov index is the number of 
half windings made by
the path in $Sp(2)$. However, in order to give a 
geometrical meaning of the Maslov 
index we need to describe some properties of the symplectic group of the plane.

The symplectic group of the plane $Sp(2)$ consists of the real matrices
two by two $A$ such that $A^{T}JA=J$, where $A^{T}$ is the transpose of $A$
and
$$
J = \left(
\begin{array}{cc}
0 & 1 \\
-1 & 0\\
\end{array}
\right).
$$
The eigenvalues $\lambda_1$ and $\lambda_2$ of $A \in Sp(2)$ are of the
following form:
 \begin{quote}
\begin{itemize}
\item
\emph{$\lambda_1=\lambda_2 =1$}
\item
\emph{$\lambda_1=\lambda_2 =-1$}
\item
\emph{$\lambda_1=\bar{\lambda_2} \quad \lambda_1, \lambda_2 \in S^1- \{+1,-1\}$}
\item
\emph{$\lambda_1=\frac{1}{\lambda_2} \quad \lambda_1, \lambda_2  \in \mathbb{R}- \{+1,-1\}$ }
\end{itemize}
\end{quote}

A parametrization of $Sp(2)$ due to Gel'fand and Lidski\v\i\
allows to visualize the simplectic group as  $S^1 \times D$ where $D$ is 
the unitary disk.

The set of matrices in $Sp(2)$ that correspond to the degenerate condition, i.e
those such that $1$ is a Floquet multiplier, disconnect the 
simplectic group into
two regions $\alpha$ and $\beta$. The set $\alpha$ is that of the matrices with
distinct and real positive Floquet multipliers and
$\beta$ is that of matrices with complex or real negative Floquet multipliers. 
Figure \ref{Figure3} gives a rough idea of sets  $\alpha$ and $\beta$ in the
symplectic
group; we refer to \cite{Abb01} for a rigorous and pretty picture.

\begin{center}
\begin{figure}
\hspace{3.5cm}
\includegraphics[scale=0.7,angle=0]{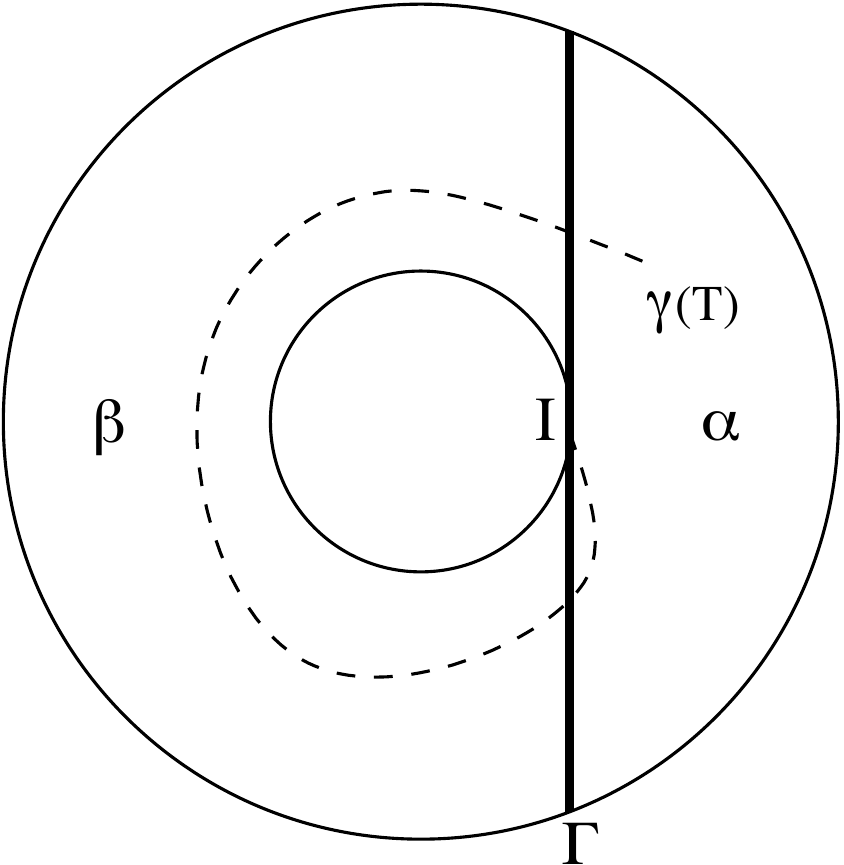}
\caption{The sets $\alpha$, $\beta$ and the set $\Gamma$ of degenerate 
matrices. The set $\Gamma$ is represented by the vertical line.}
\label{Figure3}
\end{figure}
\end{center}

Now, we can state the proposition that relates the Bott index
with the Maslov index and that characterize the periodic orbit
depending on the parity of the Maslov index.
\begin{prop}
\label{propmaslov}
The Maslov index $\mu_{\gamma}(T)$ of the path $\gamma : [0,T] \rightarrow Sp(2)$ fulfills the
following properies\\
\\
(i) $\mu_{\gamma}(T)=j(T,1)$\\
\\
(ii) $\mu_{\gamma}(T)$ is even if and only if the Floquet multipliers of $\gamma(T)$
are distinct and real positive\\
\\
(iii) $\mu_{\gamma}(T)$ is  odd if and only if the Floquet multipliers of $\gamma(T)$
are complex or real negative\\
\end{prop}
\begin{proof}
The proof of  $(i)$, $(ii)$, $(iii)$ can be found in 
\cite{An98}, \cite{Lon98} and \cite{Abb01}
\end{proof}

\section{Main Results on periodic orbits}

We want to introduce the rotation frequency of a curve as follows.

\begin{definition}
If $x(t)\in H^1_{k_1,k_2T_0}$, then its rotation frequency is
\begin{equation}
\rho(x)=\frac {k_1}{k_2T_0}
\end{equation}
\end{definition}

We recall that the Morse index allows to separate the periodic orbits in two
distinct classes as described in the previous section.
\begin{definition}
Let $x(t)$ be a periodic solution of (\ref{equation1}) in $H^1_{k_1,k_2T_0}$;
$x(t)$ is periodic of type $\alpha$ (positive distinct Floquet multipliers) if $m(x,k_2T_0)$ is
even,
and periodic of type $\beta$ if $m(x,k_2T_0)$ is odd (complex or negative Floquet multipliers).
\end{definition}
\begin{prop}
For any $x$ periodic solution in $H^1_{k_1,k_2T_0}$, we have that $f(x)$
is bounded  by a constant which depends only on $k_1$ and $k_2T_0$.
\end{prop}
\begin{proof}
The value of $|V'(t,x(t))|$ is bounded because $V'(t,x)$ is a $C^1$
function on the compact set $S^1 \times S^1$.
Thus, it is sufficient to prove that $|\dot{x}(t)|$ is
bounded for all $x$ periodic solutions in $H^1_{k_1,k_2T_0}$.
By eq. (\ref{bound}), we know that $x(t)=\rho t +y(t)$ where $y \in H^1_{0,k_2T_0}$.

Moreover, we have that
\begin{displaymath}
\int_{0}^{k_2T_0} \dot{y}(t)dt=y(k_2T_0)-y(0)=0,
\end{displaymath}
therefore, for any $y \in H^1_{0,k_2T_0}$, there exist $\bar{t}$ such that $\dot{y}(\bar{t})=0$.
For the Lagrange theorem there exist $\xi \in [0,k_2T_0]$ such that
\begin{equation}
\frac{|\dot{y}(t)- \dot{y}(\bar{t})|}{|t-\bar{t}|} = |\ddot{y}(\xi)| =
|\ddot{x}(\xi)|= |V'_{x}(\xi,x(\xi))| \leq C.
\end{equation}
So, $|\dot{y}(t)| \leq C|t-\bar{t}| \leq Ck_2T_0$.
Finally, for any periodic solution  $x(t)\in H^1_{k1,k_2T_0}$
\begin{displaymath}
|\dot{x}(t)| = |\rho +\dot{y}(t) | \leq \rho + Ck_2T_0.
\end{displaymath}
\end{proof}

\begin{prop}\label{parity}
The number of critical point of $f$ is even in $H^1_{k_1,k_2T_0}$.
\end{prop}
\begin{proof}
The functional  $f$ is bounded on the periodic solutions in $H^1_{k_1,k_2T_0}$
by the above proposition. The Palais-Smale condition and the assumption
that eq. (\ref{equation1}) is non-resonant assures that the critical
point are in a finite number.

We apply the Morse relations given by (\ref{morserelationsclass})
\begin{displaymath}
\sum_{x \text{critical}}\lambda^{m(x)}=
{\cal P}_\lambda(M)+(1+\lambda) {\cal Q}_\lambda.
\end{displaymath}
where $M$ is $H^1_{k_1,k_2T_0}$.
The lemma (\ref{basiclemma}) shows that
${\cal P}_\lambda(H^1_{k_1,k_2T_0})=1+\lambda $ and, therefore,
the Morse relation becomes
\begin{equation}
\label{morserelations}
\sum_{x \text{critical}}\lambda^{m(x)}=
1+\lambda+(1+\lambda) {\cal Q}_\lambda = (1+\lambda) {\cal \tilde{Q}}_\lambda .
\end{equation}
$\lambda =1$ shows that the number of periodic solutions are $2 {\cal \tilde{Q}}_\lambda$.
\end{proof}

\begin{definition}
Let $\rho=\frac{k_1}{k_2T_0}$. We set
\begin{eqnarray*}
n_\alpha(\tau,\rho)&=&
\left\{
\begin{array}{c}
\text{number of fundamental solutions } x \text{ of type }\alpha
\text{ with }\\
\tau(x) =\tau ;\ \rho(x)=\rho
\end{array}
\right\} \\
n_\beta(\tau,\rho)&=&
\left\{
\begin{array}{c}
\text{number of fundamental solutions } x \text{ of type }\beta
\text{ with }\\
\tau(x) =\tau ;\ \rho(x)=\rho
\end{array}
\right\}
\end{eqnarray*}
and the function
\begin{eqnarray}
\chi (\tau,\rho) &:=&n_\alpha(\tau,\rho)-n_\beta(\tau,\rho)
\end{eqnarray}
\end{definition}

\begin{rem}
For the periodic solutions with the Morse index equal to
an even number $2m$ we have
\begin{displaymath}
\tau=\frac{2m}{k_2T_0}
\end{displaymath}
while for the solutions with Morse index equal to $2m+1$ we have
\begin{displaymath}
\tau \in \left( \frac{2m}{k_2T_{0}},\frac{2m+2}{k_2T_{0}}\right).
\end{displaymath}
The Proposition \ref{propmaslov} implies that a
periodic solution $x(t)$
is periodic of type $\alpha$ iff the Floquet exponent are distinct and
real positive, i.e iff the symplectic matrix $\gamma(T)$ is in the $\alpha$
component of $Sp(2)$.
On the contrary, $x(t)$ is periodic of type $\beta$ iff the eigenvalues
are complex or real negative, i.e. if $\gamma(T)$ is in the $\beta$ component
of $Sp(2)$. The periodic orbits with an even Morse index are of kind $\alpha$,
i.e have positive distinct Floquet multipliers. By Proposition \ref{propTwist} 
we have that
the twisting frequency is the mean Morse index, so if $x(t)$ is of kind $\alpha$ we have
\begin{displaymath}
\tau = \frac{1}{2\pi T} \int_{S^1} j(T,\sigma) d\sigma=
\frac{1}{2\pi k_2T_0} \int_{S^1}2m\ d\sigma=
\frac{2m}{k_2T_0}.
\end{displaymath}
When the Morse index is an odd number we have that
$j(T,\sigma)$ is not constant but it assumes only the values $2m+1$ and $(2m+1)\pm 1$,
so we obtain the other estimate.
\end{rem}

\begin{prop}
Let $\rho=\frac{k_1}{k_2T_0}$. We have $\chi(0,\rho)>0$.
\end{prop}

\begin{proof}
For all $\rho=\frac{k_1}{k_2T_0}$,
by the previous remark we have immediately that
$\chi(\tau,\rho) \geq 0$ for $\tau=\frac{2m}{k_2T_0}$
and $\chi(\tau,\rho) \leq 0$ for $\tau \in \left( \frac{2m}{k_2T_{0}},
\frac{2m+2}{k_2T_{0}}\right)$. Furthermore,
the Morse relations given by (\ref{morserelationsclass})
and Lemma \ref{basiclemma} show that there exist
periodic orbits with Morse index 0. This concludes the proof.
\end{proof}

\begin{definition}
For all $\rho=\frac{k_1}{k_2T_0}$ we set the multiplicity functions
\begin{displaymath}
\nu (\tau,\rho) :=\sum_{\zeta \leq \tau }\chi (\zeta,\rho)
\end{displaymath}
\begin{displaymath}
\eta(\tau,\rho) :=\sum_{\zeta < \tau }\chi (\zeta,\rho)
\end{displaymath}

Clearly $\nu(\tau,\rho)=\eta(\tau,\rho)+ \chi(\tau,\rho)$.\\
These functions are well defined because $\chi(\tau,\rho) \neq 0$ only for a
finite number
of $\tau$. The total number of $\alpha$ and $\beta$ periodic solutions
are given by the Morse relations given by (\ref{morserelationsclass}).
The topology of $H^1_{k_1,k_2T_0}$ given by Lemma
\ref{basiclemma} and the Morse relations imply that the total number of
solutions with even Morse index are
equal to the number of solutions with odd Morse index.
If we call $\tau_{max}$ the maximum value of $\tau$ among the fundamental
periodic solutions, we have
$\nu(\tau, \rho)=0$ if $\tau \geq \tau_{max}$.
\end{definition}

\begin{prop}\label{eta}
Let $\rho=\frac{k_1}{k_2T_0}$, there exists $\epsilon_0 >0$ such that
\begin{eqnarray*}
\eta(\tau+\epsilon, \rho) = \nu(\tau, \rho) && \forall \ 0 < \epsilon < \epsilon_0;\\
\nu(\tau+\epsilon, \rho) = \nu(\tau, \rho) && \forall \ 0 < \epsilon < \epsilon_0.
\end{eqnarray*}
\end{prop}
\begin{proof}
Given $\rho=\frac{k_1}{k_2T_0}$, we know that there exist a finite
number of fundamental solutions. Therefore, there exists  $\epsilon_0 >0$
such that $\chi(\xi,\rho)=0$ if $\xi \in (\tau, \tau + \epsilon_0)$.
The proof follows straightforward.
\end{proof}

\begin{prop}
\label{basicprop} Let $y$ be a non-fundamental periodic solution  of
$\ddot{x}+V'(t,x)=0$ in
$C^2_{pk_1,pk_2T_0}$ with $p$ prime. Then $y(t)$, $y(t+k_2T_0)$,
$y(t+2k_2T_0)$,...,
$y(t+(p-1)k_2T_0)$ are $p$ distinct non-fundamental periodic solutions.
\end{prop}
\begin{proof}
$y$ is a periodic solution that makes $pk_1$ windings
in $pk_2T_0$ time; $y$ is non-fundamental, thus, it is
nonperiodic of period $k_2T_0$.

At first we show that $y(t+lk_2T_0)$ is a solution. We have that
\begin{eqnarray*}
\ddot y(t+lk_2T_0)+V'(t,y(t+lk_2T_0))&=&\\
=\ddot y(t+lk_2T_0)+V'(t+lk_2T_0,y(t+lk_2T_0))&=&0.
\end{eqnarray*}
Furthermore, suppose that there exists $l\neq m$ with $l,m <p$ such that
\begin{displaymath}
 y(t+lk_2T_0)=y(t+mk_2T_0).
\end{displaymath}
After a change of variables we have that
\begin{displaymath}
 y(t+(l-m)k_2T_0)=y(t)\ \forall t.
\end{displaymath}
but $p$ is prime and that contradicts our hypothesis
\end{proof}

Now, we can prove Theorem \ref{primoteo} and Theorem \ref{maintheorem}.
\begin{proof}[Proof of Theorem \ref{primoteo}]

Given any rotation frequency $\rho \in \frac{1}{T_0} \mathbb{Q}$, we take
$k_1$ and $k_2$ coprime such that $\rho=\frac{k_1}{k_2T_0}$.
If eq. (\ref{equation1}) is non-resonant we have, by Proposition \ref{parity}, an even number of periodic solutions in
$H^1_{pk_1,pk_2T_0}$ for any $p \in \N^+$.

These periodic solutions are fundamental solutions if we take $p=1$.
\\

Clearly, if $x \in H^1_{k_1,k_2T_0}$ then $x \in H^1_{pk_1,pk_2T_0}$ and the Morse index
$m(x,pk_2T_0)$ fulfills the property $(iii)$ of Proposition \ref{propTwist}
\begin{displaymath}
\tau \left( x\right) pk_2T_0-1 \leq m(x,pk_2T_0)\leq\tau \left( x\right) pk_2T_0+1.
\end{displaymath}
The Morse relations assures
that there exist $y \in H^1_{pk_1,pk_2T_0}$ with $m(y,pk_2T_0)=1$.
This orbit fulfills $\rho(y)=\frac{k_1}{k_2T_0}$ and it is non-fundamental when
$p$ is sufficently large
because it cannot be $k_2T_0$ periodic. Indeed, the property $(iii)$ of 
Proposition \ref{propTwist} assures that all the periodic orbits $x(t)$ in 
$H^1_{k_1,k_2T_0}$ with Morse index 1 have a Morse index $m(x,pk_2T_0) > 1$
when $p$ is sufficently large.

Moreover, the periodic orbits $x(t)$ in 
$H^1_{k_1,k_2T_0}$ with Morse index 0 have a Morse index $m(x,pk_2T_0) = 0$
for the same reason.

This proves that, taken $p$  sufficently large, the periodic orbit 
$y \in H^1_{pk_1,pk_2T_0}$ with $m(y,pk_2T_0)=1$ cannot be $k_2T_0$ periodic and therefore it is non-fundamental.

Hence, there exist infinitely many non-fundamental orbits with
$\rho=\frac{k_1}{k_2T_0}$.

\end{proof}

\begin{proof}[Proof of Theorem \ref{maintheorem}]

Without any lack of generality we demonstrate the theorem for $k_1=k$ and
$k_2=1$. The generalization
to $\rho=\frac{k_1}{k_2T_0}$ is straightforward. We consider, therefore,
the case $\rho=\frac{k}{T_0}$. Moreover, in order to avoid a too heavy notation
we will use $\nu(\tau)$, $\chi(\tau)$  and $\eta(\tau)$ instead of
$\nu(\tau,\rho)$, $\chi(\tau,\rho)$ and $\eta(\tau,\rho)$. All this functions
have to be considered, however, depending on $\rho$.\\
The leading idea for these results is that a $T_0$ periodic solution
$x(t)\in H^1_{k,T_0}$ is also a $pT_0$ periodic solution. In this case
we can consider $x\in H^1_{pk,pT_0}$.

The Morse relations (\ref{morserelationsclass}) for the $pT_{0}$-periodic
solutions may be written in the following way

\begin{displaymath}
\sum_{j}a_j\lambda^j=
1+\lambda+(1+\lambda) {\cal Q}_\lambda = (1+\lambda) \sum_{j} q_j\lambda^j .
\end{displaymath}
with a compact notation
\begin{eqnarray}
a_{0} &=&q_{0}  \label{morse1}\\
a_{j} &=&q_{j}+ q_{j-1}  \nonumber
\end{eqnarray}
or in a non compact form
\begin{eqnarray}
q_{0} &=&a_{0} \nonumber   \\
q_{1} &=&a_{1}-a_{0}  \nonumber \\
q_{2} &=&a_{2}-a_{1}+a_{0} \label{morse2} \\
\ldots &=&\ldots\ldots\ldots  \nonumber \\
q_{2n} &=&a_{2n}-\ldots\ldots\ldots-a_{1}+a_{0}  \nonumber \\
q_{2n+1} &=&a_{2n+1}-\ldots\ldots\ldots+a_{1}-a_{0}  \nonumber
\end{eqnarray}

Let us consider the Modular arithmetic given by the function
$\left[ \cdot \right] :{\Z\rightarrow }{\Z}_{p},$. For any
$a_{j}$, Proposition \ref{basicprop} implies that
\[
\left[ a_{j}\right] =\left[ \alpha _{j}\right]
\]
where $\alpha _{j}$ is the number of the $pT_{0}$-periodic solutions with
Morse index $j$ that
are fundamental solutions.\\
If $j$ is even, the $\alpha _{j}$  fundamental solutions  $x(t)$ are of
kind $\alpha$ and we have
\[
j=pm(x,T_{0})=p\tau \left( x\right) T_{0}.
\]
We have
\[
\tau \left( x\right) =\frac{j}{pT_{0}}
\]
and
\[
\alpha _{j}=\chi \left( \tau \left( x\right) \right) =
\chi \left( \frac{j}{pT_{0}}\right) .
\]

If $j$ is odd, the $\alpha _{j}$ fundamental periodic solutions $x(t)$
are of kind $\beta $ and we have
\[
\tau \left( x\right) pT_0-1<j<\tau \left( x\right) pT_0+1
\]
and, hence,
\begin{displaymath}
\frac{j-1}{pT_{0}} < \tau \left( x\right) <\frac{j+1}{pT_{0}}
\end{displaymath}

\[
\tau \left( x\right) \in \left( \frac{j-1}{pT_{0}},\frac{j+1}{pT_{0}}\right)
=\frac{1}{pT_{0}}\left( j-1,j+1\right).
\]
We obtain
\[
\alpha _{j}=
- \sum_{\tau \in \left( \frac{j-1}{pT_{0}},\frac{j+1}{pT_{0}}\right) }
\chi \left( \tau \right)
\]

If we take $p$ prime and we use the
Modular arithmetics, the Morse relations
(\ref
{morse2}) becomes
\begin{eqnarray*}
\left[ q_{0}\right] &=&\left[ \alpha _{0}\right] = \chi \left( 0\right) =
\nu(0)
\\
\left[ q_{1}\right] &=&\left[ \alpha _{1}\right] -\left[ \alpha _{0}\right] =
\left[ - \sum_{\tau \in\left( \frac{0}{pT_{0}},\frac{2}{pT_{0}}\right) }
\chi \left( \tau \right)\right]- \left[\chi \left( 0\right)\right]=
\left[ \chi \left( \frac{2}{pT_{0}}\right) -\nu
\left( \frac{2}{pT_{0}}\right) \right]
\\
\left[ q_{2}\right] &=&\left[ \alpha _{2}\right] -\left[ \alpha _{1}\right] +
\left[ \alpha _{0}\right] =\left[ \chi \left( \frac{2}{pT_{0}}\right) \right]
-\left[- \sum_{\tau \in \left( \frac{0}{pT_{0}},\frac{2}{pT_{0}}\right) }
\chi\left( \tau \right) \right] +
\left[ \chi \left( 0\right) \right]
\\
&=&\left[ \nu \left( \frac{2}{pT_{0}}\right) \right]
\\
\ldots &=&\ldots\ldots\ldots
\\
\left[ q_{2n}\right] &=&\left[ \alpha _{2n}\right] -\ldots\ldots\ldots  -
\left[ \alpha_{1}\right] +
\left[ \alpha _{0}\right] =
\left[ \nu \left( \frac{2n}{pT_{0}}\right) \right]
\\
\left[ q_{2n+1}\right] &=&\left[ \alpha _{2n+1}\right] -\ldots\ldots \ldots +
\left[\alpha _{1}\right] -\left[ \alpha _{0}\right] =
\left[ \chi \left( \frac{2n+2}{pT_{0}}\right) -
\nu \left( \frac{2n+2}{pT_{0}}\right) \right].
\end{eqnarray*}

The periodic solutions with Morse index $2n$
are of type $\alpha$ and with twisting frequency $\tau= \frac{2n}{pT_0}$ .

We have
\begin{displaymath}
\left[ q_{2n}\right] =\left[ \nu \left( \frac{2n}{pT_{0}}\right) \right].
\end{displaymath}
If $[\nu \left( \frac{2n}{pT_{0}}\right)] \neq 0$, we have $\left[ q_{2n}\right] \neq 0$ and, therefore,
$q_{2n}\neq 0$ and $a_{2n} \geq [\nu \left( \frac{2n}{pT_{0}}\right)]$.
\\

On the other hand, the periodic solutions with twisting frequency 
$\tau$ such that 
$|\tau - \frac{2n+1}{pT_0}| < \frac{1}{pT_0}$
are of type $\beta $ with Morse index $2n+1$.

We have $$\left[ q_{2n+1}\right]= \left[ \chi \left( \frac{2n+2%
}{pT_{0}}\right) -\nu \left( \frac{2n+2}{pT_{0}}\right) \right] = [-\eta(\frac{2n+2}{pT_0})].$$ If  $[- \eta(\frac{2n+2}{pT_0})] \neq 0$,  we have $\left[ q_{2n+1}\right] \neq 0$ and, therefore,
$q_{2n+1}\neq 0$ and $a_{2n+1} \geq [- \eta(\frac{2n+2}{pT_0})]$.
\\
\\
We have demonstrated that for all $\tau \in \mathbb{N}/(pT_0)$, $p$ prime,
there exist orbits with twisting frequency arbitrary close to $\tau$.

In particular, if $\tau=\frac{2n}{pT_0}$, there exist at least $[\nu(\tau)]$
 solutions $x(t)$ of type $\alpha$ and $pT_0$-periodic such that
\begin{displaymath}
\tau(x)=\tau
\end{displaymath}

On the other hand, if $\tau=\frac{2n+1}{pT_0}$, there exist at least
$[-\eta(\frac{2n+2}{pT_0})]$  solutions $y(t)$ of type $\beta$
and $pT_0$-periodic
such that
\begin{displaymath}
|\tau(y)-\tau|<\frac{1}{pT_0}.
\end{displaymath}

\end{proof}
\begin{rem}
Theorem \ref{maintheorem} gives a lower bound on the number of solutions in $C_{pk_1,pk_2T_0}$. Notice that the periodic solutions 
are non-fundamental any time we choose $\tau$ sufficently far from any twisting
number of the fundamental solutions. In this case, the non-fundamental periodic solutions are at least $p$ by Proposition (\ref{basicprop}).
\end{rem}

As a consequence of Theorem \ref{maintheorem} we can prove the following corollary.

\begin{cor}
Let $\rho=\frac{k_1}{k_2T_0}$ and $\tau$ such that $\nu(\tau,\rho) \neq 0$.\\
Then, there exist a sequence of non-fundamental orbits
$x_n$  of type $\alpha$ and a sequence of non-fundamental orbits
$y_n$  of type $\beta$ such that
$$\tau(x_n) \rightarrow \tau$$
$$\tau(y_n) \rightarrow \tau.$$
\end{cor}
\begin{proof}
Given any $\tau$, we can choose two approximations of $\tau$ of the following
form: $\tau_{\alpha}=\frac{2n}{pk_2T_0}>\tau$ and 
$\tau_{\beta}=\frac{2n+1}{pk_2T_0}> \tau$.

We know, by Proposition \ref{eta}, that $\nu(\tau_{\alpha},\rho)= \nu(\tau,\rho)\neq 0$ and
$\eta(\tau_{\beta},\rho)=\eta(\tau,\rho)\neq 0$ when $\tau_{\alpha}$ and $\tau_{\beta}$ are
sufficiently close to $\tau$, i.e definitely for $p$ large. Therefore, we can 
choose $p$ such that $\nu(\tau_{\alpha},\rho) \neq 0$ (mod $ p$) and
 $-\eta\left(\frac{2n+2}{pk_2T_0},\rho\right) \neq 0$  
(mod $ p $). 
By Theorem \ref{maintheorem}, we have at least
$\nu(\tau_{\alpha},\rho)$ (mod $ p$) orbits with twisting frequency
$\tau_{\alpha}$ and $-\eta\left(\frac{2n+2}{pk_2T_0},\rho\right)$  
(mod $ p $) 
orbits with twisting frequency close to $\tau_{\beta}$.

Thus, if we take a sequence of $\tau_{\alpha}$ and $\tau_{\beta}$ which converges
to $\tau$ we find a sequence of orbits with twisting frequency
that converges to $\tau$. We can choose these orbits to be non-fundamental
because the fundamental orbits are in a finite number. 
\end{proof}
By this corollary we can prove the last result claimed in the introduction.

\begin{proof}[Proof of Corollary \ref{coroll}]
For any $(\tau,\rho) \in \Sigma$, we can choose a sequence $\rho_k\rightarrow \rho$ and a
sequence $\tau_k\rightarrow \tau$ such that, for all $k$, $\nu(\tau_k,\rho_k)\neq 0$. So,
by the previous corollary,
we can find
two sequence of non-fundamental orbits $x_n^k$ and $y_n^k$ such that
\begin{eqnarray*}
\tau(x_n^k) \rightarrow \tau_k&&\rho(x_n^k) \rightarrow \rho_k;\\
\tau(y_n^k) \rightarrow \tau_k&&\rho(y_n^k) \rightarrow \rho_k.
\end{eqnarray*}
A diagonal argument proofs the corollary.
\end{proof}

\begin{prop}
For all $\rho=\frac{1}{T_0}\mathbb{Q}$, let $x_n$ be
a sequence of non-fundamental orbits such that $\tau(x_n) \rightarrow \tau$, then
$x_n \rightarrow x$ in $C^1_{loc}$.
\end{prop}
\begin{proof}
Let $\rho=\frac{k_1}{k_2T_0}$ and $x_n$ be the non-fundamental orbits
with $\rho(x_n)=\rho$ and $\tau(x_n) \rightarrow \tau$.
The orbit $x_n$ makes $k_{1,n}$ windings of $S^1$ in $k_{2,n}$ time, with
$\frac{k_{1,n}}{k_{2,n}T_0}=\rho$.

The orbits $x_n$ are solution of eq.(\ref{equation1}) and $\ddot{x}_n$
is bounded by the maximum value of $|V'(t,x_n(t))|$ which is a $C^1$ function on the
compact set $S^1 \times S^1$. \\
In order to prove that $x_n \rightarrow x$ in $C^1_{loc}$
we want to show that, fixed a finite interval of time $I=[0,D]$, $x_n \in W^{2,\infty}(I).$
It is sufficient to prove
that $\dot{x}_n(0)$ is bounded. Indeed,
$$
\dot{x}_n(t)=\dot{x}_n(0)+\int_{0}^{t}-V'(x(s),s)ds.
$$
The right-hand side is bounded in $I$ iff $\dot{x}_n(0)$ is bounded.\\
By eq. (\ref{bound}), $x_n(t)=\rho t + y_n(t)$ with
$y_n \in H^1_{0,k_{2,n}T_0}$.
We have that $\dot{x}_n(t)=\rho+\dot{y}_n(t)$.
The function $y_n(t)$ is periodic, therefore $\dot{y}_n(\xi_n)=0$ for
$\xi_n \in [0,k_{2,n}T_0]$. Proposition \ref{basicprop} shows that $x_n(t)$,
$x_n(t+k_2T_0)$,..., are distinct non-fundamental periodic orbits with
the same $\tau$.
We can shift the orbits such a way that $\xi_n \in [0,k_2T_0]$.\\
All the orbits $x_n$ have a point $\xi_n \in [0,k_2T_0]$ where the derivative
is zero, therefore they should have bounded initial velocity
by means of the Lagrange theorem. Indeed
$$
\left|\frac{\dot{y}_n(\xi_n)-\dot{y}_n(0)}{k_2T_0}\right|
\leq \left|\frac{\dot{y}_n(\xi_n)-\dot{y}_n(0)}{\xi_n}\right| \leq const.
$$
Thus, $x_n \in W^{2,\infty}(I)$ which is embedded with a compact embedding
in $C^1(I)$.
\end{proof}

The authors would like to express thanks to Alberto Abbondandolo for
fruitful discussions in the preparation of the paper.

\end{document}